\documentclass[a4paper,sort&compress]{amsart}
\usepackage{amssymb}
\usepackage{amsthm,amsmath}
\usepackage{enumerate}
\usepackage{hyperref}
\usepackage{mathrsfs}
\usepackage{pgfplots}
\usepackage{array}

\usepackage{algorithm}
\usepackage{booktabs}
\usepackage{blkarray}

\usepackage[sort,compress]{cite}
\usepackage[capitalize]{cleveref}
\Crefname{equation}{}{}

\usepackage{xcolor}
\usepackage{tikz}
\usetikzlibrary{backgrounds,patterns,matrix,calc,arrows}

\title[Generic identifiability of the Chow variety]{Almost all subgeneric third-order Chow decompositions are identifiable}
\date{}

\author{Douglas A. Torrance}%\fnref{cor2}}
\thanks{}

\author{Nick Vannieuwenhoven}%\fnref{cor1}
\thanks{The resources and services used in this work were provided by the VSC (Flemish Supercomputer Center), funded by the Research Foundation---Flanders (FWO) and the Flemish Government.\\[5pt] 
Douglas A. Torrance (\texttt{dtorrance@piedmont.edu}).\\
Piedmont University, Georgia, United States of America.\\[3pt]
Nick Vannieuwenhoven (\texttt{nick.vannieuwenhoven@kuleuven.be}).\\ 
KU Leuven, Department of Computer Science, Leuven, Belgium.\\
N.V.~was supported by a Postdoctoral Fellowship of the Research Foundation---Flanders (FWO) with project 12E8119N}

\newcommand{\NN}{\mathbb N}
\newcommand{\PP}{\mathbb P}

\newcommand{\ZZ}{\mathbb Z}
\newcommand{\CC}{\mathbb C}
\newcommand{\RR}{\mathbb R}
\newcommand{\bbk}{\Bbbk}
\newcommand{\Var}[1]{\mathcal{#1}}
\newcommand{\SFF}{\mathit{I\!I}}
\newcommand{\Tang}[2]{\mathrm{T}_{#1} {#2}}
\newcommand{\Norm}[2]{\mathrm{N}_{#1} {#2}}

\DeclareMathAlphabet{\pzc}{OT1}{pzc}{m}{it}

\DeclareMathOperator{\CV}{CV}

\newtheorem{theorem}{Theorem}[section]
\newtheorem{proposition}[theorem]{Proposition}
\newtheorem{lemma}[theorem]{Lemma}

\theoremstyle{definition}

\newtheorem{remark}[theorem]{Remark}

\numberwithin{equation}{section}

\begin{document}

\begin{abstract}
For real and complex homogeneous cubic polyomials in $n+1$ variables, we prove that the Chow variety of products of linear forms  is generically complex identifiable for all ranks up to the generic rank minus two. By integrating fundamental results of [Oeding, Hyperdeterminants of polynomials, Adv. Math., 2012], [Casarotti and Mella, From non defectivity to identifiability, J. Eur. Math. Soc., 2021], and [Torrance and Vannieuwenhoven, All secant varieties of the Chow variety are nondefective for cubics and quaternary forms, Trans. Amer. Math. Soc., 2021] the proof is reduced to only those cases in up to $103$ variables. These remaining cases are proved using the Hessian criterion for tangential weak defectivity from [Chiantini, Ottaviani, and Vannieuwenhoven, An algorithm for generic and low-rank specific identifiability of complex tensors, SIAM J. Matrix Anal. Appl., 2014]. We also establish that the smooth loci of the real and complex Chow varieties are immersed minimal submanifolds in their usual ambient spaces.
\end{abstract}

\maketitle

\section{Introduction}

Let $\bbk$ denote either the reals $\RR$ or complex numbers $\CC$.
A \textit{Chow decomposition} over $\bbk$ expresses a homogeneous polynomial $p \in S^d(\bbk^{n+1})$ of degree $d$ in $n+1$ variables as a minimum-length sum of \textit{completely decomposable forms} (i.e., products of linear forms): 
\begin{align}\tag{CD}\label{eqn_CD}
 p(x_0, \ldots, x_n) 
 = \sum_{j=1}^r \prod_{i=0}^{d-1} L_{i,j} 
 = \sum_{j=1}^r \prod_{i=0}^{d-1} (a_{0,i,j} x_0 + a_{1,i,j} x_1 + \cdots + a_{n,i,j} x_n),
\end{align}
where $L_{i,j} = a_{0,i,j} x_0 + a_{1,i,j} x_1 + \cdots + a_{n,i,j} x_n$ are the $\bbk$-linear forms. The minimal number $r$ is called the \textit{Chow rank} of $p$ over $\bbk$. This decomposition and its algebro-geometric properties were previously studied in \cite{GKZ1994,AB2011,Shin2011,Shin2012,Torrance2013,Abo2014,Torrance2017,CCGO2017,QML2017,TV2020,Rodriguez2021}. Note that \cref{eqn_CD} is a depth-$3$ $\Sigma\Pi\Sigma$ arithmetic circuit, and finding high-rank Chow decompositions of a certain class of polynomials is a central problem in arithmetic complexity theory for separating the algebraic complexity classes $\mathrm{VP}$ and $\mathrm{VNP}$ \cite{BCS1997,SY2010,Landsberg2012}.

Two natural questions about complex Chow decompositions are the following:
\begin{enumerate}
 \item what is the minimal value of $r$ so that almost all polynomials $p \in S^d \CC^{n+1}$ admit a Chow decomposition of rank $r$, and 
 \item how many distinct complex Chow decompositions (up to permutation of the summands) does a generic rank-$r$ polynomial $p$ have?
\end{enumerate}
It suffices to study the complex case because analogous answers for real Chow decompositions can be derived from the answers in the complex case. Recall that a rank $r$ is called \emph{typical} if there is a Euclidean-open subset of polynomials in $S^d(\RR^{n+1})$ with this Chow rank. Blekherman and Teitler \cite{BT2015} showed that the smallest \emph{typical} real rank equals the answer to the first question. As for the second question, the observation of \cite{COV2017b,QCL2016} that generic complex $r$-identifiability also entails generic real $r$-identifiability applies because the set of bounded rank Chow decompositions is the image of the regular map implicitly defined by \cref{eqn_CD}. Consequently, in the remainder of this paper, the focus is on answering the foregoing questions for complex Chow decompositions.

Before trying to answer these questions, recall that the Chow decomposition is a generalization of another famous polynomial (or, equivalently, symmetric tensor) decomposition: by taking $L_{0,j}=\cdots=L_{d-1,j}$ in \cref{eqn_CD} we obtain the \textit{Waring} \cite{IK1999,Landsberg2012} or \textit{symmetric tensor rank} decomposition \cite{CGLM2008}, which was already studied by Clebsch, Sylvester, Palatini, and Terracini in the 19th and first half of the 20th century; see \cite{BO2008} for historical remarks. 
After a century-long journey in projective algebraic geometry starting in earnest with Palatini's 1903 paper \cite{Palatini1903}, the necessary tools, such as those in \cite{Terracini1911,AH1995,CC2001,CC2006,CO2012,CM2019}, were developed to study foregoing questions. For the Waring decomposition, the first question was answered by Alexander and Hirschowitz \cite{AH1995} in 1995, and the second was completely resolved by 2019 through the combined works of Ballico \cite{Ballico2005}, Chiantini, Ottaviani, and Vannieuwenhoven \cite{COV2017}, and Galuppi and Mella \cite{GM2017}.

As for the Chow decomposition, the picture is not yet complete. The Zariski closure of the completely decomposable forms in $n+1$ variables and degree $d$ is
\[
 \Var{C}_{d,n} = \overline{\{ [L_0\cdots L_{d-1}] \mid L_i \in \CC^{n+1} \}} \subset S^d \PP^n,
\]
where $\PP^n$ is the $n$-dimensional projective space over $\CC$.\footnote{For concreteness we present the results over $\CC$, but by the Lefschetz principle it can be substituted by any algebraically closed field of characteristic zero.} 
It is an irreducible, nondegenerate (not contained in a hyperplane), singular, projective variety called either the \textit{Chow variety of zero-cycles} \cite{GKZ1994} or the \textit{split variety} \cite{AB2011,Abo2014}.
The Chow variety is part of a larger family of subvarieties of $S^d\PP^n$ which have become known as \textit{Chow-Veronese varieties}. For any partition $\mathbf d = (d_1,\ldots,d_k) \vdash d$, $\CV_{\mathbf d}$ is the Zariski closure of the set of forms that can be written as $[L_1^{d_1}\cdots L_k^{d_k}]$ for $L_i\in\CC^{n+1}$.  In particular, the Veronese variety, which is related to Waring decomposition, is $\CV_{(d)}$ and the Chow variety $\Var C_{d,n}$ is $\CV_{(1,\ldots,1)}$.

The $d$-forms $[p] \in S^d \PP^n$ such that $p$ admits a rank-$r$ Chow decomposition \cref{eqn_CD} form a Zariski-open subset of the \textit{$r$-secant variety} $\sigma_r(\Var{C}_{d,n}) \subset S^d \PP^n$. Recall that for an irreducible, nondegenerate, projective variety $\Var{V}\subset\PP^N$ the $r$-secant variety is defined as the image of
the \textit{abstract secant variety} $\Sigma_r(\Var{V}) = \overline{\{[\sum_{i=1}^r p_i], ([p_1],\ldots,[p_r]) \}} \subset \PP^N \times \Var{V}^{\times r} $, which is a smooth projective variety of dimension $r(1+\dim\Var{V})-1$, under the
projection map 
\[
\pi_r^{\Var{V}} : \PP^N \times \Var{V}^{\times r}  \to \PP^N;
\]
see, e.g., \cite{Russo2016,Zak1993}.
In \cite{TV2020}, we answered the first question for Chow decompositions of cubics, showing that $\sigma_r(\Var{C}_{3,n})$ is, without exceptions, \textit{nondefective}, which means that $\dim \sigma_r(\Var{C}_{3,n}) = \min\{N, \dim \Sigma_r(\Var{C}_{3,n})\}$ for all $n \in \NN$. This entails that a generic cubic $p \in S^3 \CC^{n+1}$ has rank equal to $r_\text{gen} = \lceil \frac{1}{3n+1} \binom{n+3}{3} \rceil$ for all $n$.

The main contribution of this paper consists of leveraging the main results of \cite{CM2019,TV2020,COV2014,Oeding2012} to tackle the second question on the number of Chow decompositions \cref{eqn_CD} for cubics. In particular, we prove that for almost all subgeneric ranks $r \le r_\text{gen}$ there exists a Zariski-open subset of $\sigma_r(\Var{C}_{3,n})$ such that the cubics in that set admit a \textit{unique} expression as in \cref{eqn_CD}, up to the order of the summands. This is called the \textit{generic $r$-identifiability} of $\Var{C}_{3,n}$. We will thus prove the following result. 

\begin{theorem} \label{thm_main}
Let $\Var{C}_{3,n} \subset S^3 \PP^{n}$ be the degree-$3$ Chow variety of zero-cycles. Then, $\Var{C}_{3,n}$ is generically complex $r$-identifiable up to the generic rank minus $2$:
\[
 r \le r_\text{gen} - 2 = \left\lceil \frac{1}{3n+1} \binom{n+3}{3} \right\rceil - 2.
\]

In other words, the complex Chow decomposition \cref{eqn_CD} of a generic real or complex Chow rank-$r$ cubic $p$ is unique in the sense that the set of completely decomposable forms $\{\prod_{i=0}^{d-1} L_{i,1}, \dots, \prod_{i=0}^{d-1} L_{i,r}\}$ is uniquely determined by $p$.
\end{theorem}
\begin{remark}

We prove generic identifiabilty under the marginally stronger bound 
 \[
  r \le \left\lfloor \frac{1}{3n+1} \binom{n+3}{3} \right\rfloor - 1.
 \]
This extends \cref{thm_main} to $r \le r_\text{gen}-1$ in the three additional \emph{perfect cases} where $\frac{1}{3n+1} \binom{n+3}{3}$ is an integer, namely $n = 1$, $3$, and $13$. Generic identifiability in these cases is established by computer in \cref{sec_base_cases}.
\end{remark}

This result is almost optimal in the sense that generic $r$-identifiability holds for \emph{at most one additional value} of $r$. Indeed, (generic) $r$-identifiability fails once $r \ge \lfloor \frac{1}{3n+1}\binom{n+3}{3} \rfloor + 1$ because then the projection map $\pi_r^{\Var{C}_{3,n}}$ has strictly positive-dimensional fibres as $r (1 + \dim{} \Var{V}) - 1 > N$.
 
 The main theorem is essentially proved as follows. Since $r$-identifiability implies $k$-identifiability for all $k\leq r$, we need only worry about the upper bound. We then exploit the Casarotti--Mella theorem \cite{CM2019}, which connects $r$-nondefectivity and $(r-1)$-identifiability under a mild technical condition, called not \textit{$1$-tangential weak defectivity} ($1$-twd); see the next section for the precise definition. This not $1$-twd property of a projective variety is equivalent to its \textit{dual variety} being a hypersurface \cite{COV2017}. Oeding \cite{Oeding2012} studied the dimensions of dual varieties of Chow--Veronese varieties and established among others that $\Var{C}_{3,n}$ is, in our terminology, not $1$-twd for $n\ge2$. We already proved in \cite{TV2020} that $\Var{C}_{3,n}$ is always nondefective. Therefore, with all the foregoing observations, the proof is reduced to a finite number of cases in finite dimension that can be treated with the Hessian criterion \cite{COV2014} to conclude generic $r$-identifiability of \cref{eqn_CD} for almost all ranks for cubics, i.e., $d=3$. 
 
Note that $n\leq 1$ is not covered by the foregoing argument, but no positive $r$ satisfies the inequality from \cref{thm_main} in these cases, so we need not worry about them for its proof.  Nevertheless, since every unary $(n=0)$ form is a monomial and every binary $(n=1)$ form may be factored uniquely into a product of linear forms by the fundamental theorem of algebra, it follows that $\Var C_{d,n}$ is everywhere $1$-identifiable for all $d$ when $n\leq 1$, no genericity required.

In the next section we prove \cref{thm_main}. Thereafter, in \cref{sec_chow_not1wd}, we investigate some Riemannian geometry of the real and complex Chow variety. In particular, we compute its \textit{second fundamental form}, hereby (i) establishing the minimality of the smooth locus of the real Chow variety as a Riemannian immersed submanifold of $S^d(\RR)$ and (ii) furnishing an alternative proof of its not $1$-twd property via the Katz dimension formula. The second fundamental form naturally appears in the expression of the Riemannian Hessian of the squared distance function from the smooth locus of the Chow variety. Hence, it can be used in Riemannian quasi-\-Newton optimization algorithms over (products of) the smooth loci of Chow varieties, similar to \cite{Dirckx2019} for the Waring decomposition.

\subsection*{Acknowledgements} We thank Giorgio Ottaviani for reminding us of the connection between dual varieties and tangential weak defectivity and pointing us to Oeding's result \cite[Theorem 1.3]{Oeding2012}.

\section{Generic identifiability of cubic Chow decompositions}

We prove \cref{thm_main} by integrating several known results in \cref{sec_general_cases} that reduce the proof to only a finite number of cases. These remaining cases are treated in \cref{sec_base_cases} by a special-purpose computer program.

\subsection{Proof in the case of many variables}\label{sec_general_cases}
Chiantini and Ciliberto \cite{CC2001} introduced the concept of \emph{$r$-weak defectivity} of a projective variety $\Var{V} \subset \PP^N$ initially as a technique for studying the defectivity of $\Var{V}$'s $r$th secant variety 
\(\sigma_r(\Var{V})\).\footnote{Note that $r$ in this paper equals the number of points, in contrast to the notation used in some papers where $r$ equals the projective dimension of the subspace spanned by these points.}
Recall from \cite{CC2001} that an irreducible, nondegenerate projective variety $\Var{V}\subset \PP^N$ is $r$-weakly defective if the generic hyperplane tangent to $\Var{V}$ at $r$ points is tangent along a positive-dimensional subvariety of $\Var{V}$. Chiantini and Ciliberto later showed in \cite{CC2006} that not $r$-weak defectivity implies generic $r$-identifiability. That is, the generic point $p \in \sigma_r(\Var{V})$ admits only one expression as a linear combination of $r$ elements from $\Var{V}$. 

A powerful sufficient condition for generic $r$-identifiability of $\Var{V}$ was introduced by Chiantini and Ottaviani in \cite{CO2012} and applied to the Segre variety. A variety $\Var{V}$ is called $r$-twd if for $r$ generic, smooth points $p_i \in \Var{V}$, we have that the \textit{$r$-tangential contact locus}
\[
 \{ x \in \Var{V} \mid \Tang{x}{\Var{V}} \subset \langle \Tang{p_1}{\Var{V}}, \ldots, \Tang{p_r}{\Var{V}} \rangle \}
\]
has a positive-dimensional component \cite{CO2012}; herein, $\Tang{x}{\Var{V}}$ denotes the Zariski tangent space to $\Var{V}$ at $x$, and $\langle \cdot, \dots, \cdot \rangle$ denotes the linear span.

We have the following chain of implications of generic properties of a projective variety $\Var{V}$:
\[
\text{not $r$-weak defectivity} \Longrightarrow \text{not $r$-twd} \Longrightarrow \text{$r$-identifiability} \Longrightarrow \text{$r$-nondefectivity},
\]
where the first implication is by definition, the second by \cite[Proposition 2.4]{CO2012}, and the third essentially by definition (see also \cite{CC2006}).
Casarotti and Mella \cite{CM2019} established a partial converse to this chain of implications. They showed that not $(r-1)$-twd is also implied by $r$-nondefectivity if $\Var{V}$ is not $1$-twd and $r$ is sufficiently large. The precise statement we exploit is as follows.

\begin{theorem}[Casarotti and Mella \cite{CM2019}] \label{lem_CM}
 Let $\Var{V} \subset \PP^n$ be an irreducible, nondegenerate projective variety that is not $1$-twd. If $r > 2 \dim \Var{V}$ and the projection map $\pi_{r}^{\Var{V}}$ is generically finite, then $\Var{V}$ is $(r-1)$-identifiable.
\end{theorem}

It is known to the experts that a projective variety $\Var{V}\subset\PP^N$ is not $1$-weakly defective (and hence not $1$-twd) if and only if its \emph{dual variety} $\Var{V}^\vee$ is a hypersurface; this is stated explicitly in \cite[Proposition 4.1]{COV2017}. In the case of Chow--Veronese varieties, the dimension of the dual variety was computed by Oeding in Theorem 1.3 of \cite{Oeding2012}. From this result follows that Chow varieties $\Var{C}_{d,n}$ with $d \ge 3$ and $n \ge 2$ are not $1$-weakly defective.\footnote{As communicated to us by Giorgio Ottaviani, there appears to be a small typo in the statement of \cite[Theorem 1.3]{Oeding2012}, namely $d \ge 2$ should be $d \ge 3$. In \cref{sec_chow_not1wd}, we provide an alternative, elementary proof when $3 \le d \le n+1$ based on the Katz dimension formula \cite{GKZ1994}.}

We proved in \cite{TV2020} that $\Var{C}_{3,n}$ is never defective, so \cref{lem_CM} can be applied if $2 \dim \Var{C}_{3,n} < \lfloor \frac{1}{3n+1} \binom{n+3}{3}\rfloor$. Since $\dim\Var C_{3,n} = 3n + 1$, this inequality is only satisfied if
\(
 n \ge 103.
\)
Consequently, the task of proving \cref{thm_main} has been reduced to proving identifiability of the Chow varieties $\Var{C}_{3,n}$ with $n \le 102$. This is settled in the next subsection.

\begin{remark}
In \cite{TV2020}, we also proved that $\Var C_{d,3}$ is never defective, but the only additional case where \cref{prop_n1wd_chow} applies is when $d=4$.  Applying \cref{lem_CM} could then establish at most $1$-identifiability, which holds trivially.
\end{remark}

\begin{remark}
In \cite{AV2018}, the nondefectivity of the $(d-1,1)$-Chow--Veronese variety was proved. Hence it would be tempting to apply the same reasoning in this setting as well. Unfortunately, as proved by Oeding \cite{Oeding2012}, this variety is, in fact, $1$-weakly defective so that Casarotti and Mella's result cannot be applied directly.
\end{remark}

\subsection{Computer proof for the remaining cases in few variables} \label{sec_base_cases}

For concluding the proof of \cref{thm_main}, $r$-identifiability of the remaining Chow varieties $\Var{C}_{3,n}$ with $2 \le n \le 102$ is verified by a computer algorithm, described next. 

Recall that not $r$-twd implies not $k$-twd for all $k \le r$, so it suffices to check the case $r = \lceil \frac{1}{3n+1} \binom{n+3}{3} \rceil - 1$. 
A method for verifying not $r$-twd was described in \cite{BCO2014,COV2014} and applied to the Segre variety. The Hessian criterion \cite[Section 2]{COV2014} is a particularly efficient implementation that can be applied to any nondegenerate projective variety $\Var{V}\subset\PP^N$ whose cone has a polynomial parameterization $f : \CC^n \to \CC^{N+1}$ and whose $r$-secant variety is not defective. Indeed, it can be verified that the methodology from \cite[Section 2]{COV2014} applies verbatim in this setting. Applying it to the Chow variety $\Var{C} = \Var{C}_{d,n}$, we obtain the following algorithm:
\begin{enumerate}
 \item[S1.] Choose $r$ generic points $[p_j] \in {\Var{C}}$ such that $[P] = [p_1 + \cdots + p_r]$ is a smooth point of $\sigma_r(\Var{C}).$
 \item[S2.] Construct the tangent space $T = \langle \Tang{p_1}{\widehat{\Var{C}}}, \ldots, \Tang{p_r}{\widehat{\Var{C}}} \rangle = \Tang{P}{\sigma_r(\widehat{\Var{C}})} \subset S^{d} \CC^{n+1}$. Equality holds because of Terracini's lemma \cite{Terracini1911}. 
 \item[S3.] Let $N = (n_1, \ldots, n_q) \subset S^d \CC^{n+1}$ be a basis for the orthogonal complement of $\Tang{P}{\sigma_r(\widehat{\Var{C}})}$ in the Hermitian inner product. Note that the span $\langle N \rangle$ is the normal space $\Norm{P}{\sigma_r(\widehat{\Var{C}})}$. Choose a generic element $\eta \in \Norm{P}{\sigma_r(\widehat{\Var{C}})}$.
 \item[S4.] Choose any $p_j$ and let the second fundamental form of $\widehat{\Var{C}}$ at $p_j$ be $\SFF_{p_j}$. Then, compute the ``Hessian matrix'' $H_{j} := \SFF_{p_j}^*(\eta)$.
 Practically this can be accomplished by computing the components of the Hessian ``matrix'' of the parameterization  
 \[
 \CC^{n+1}\times\CC^{n+1}\times\CC^{n+1} \mapsto S^3 \CC^{n+1}, (L_1, L_2, L_3) \mapsto L_1 L_2 L_3
 \]
 which are elements of $S^d \CC^{n+1}$, and contracting them with $\eta$ to obtain $H_{j}$.
\end{enumerate}

Note that step S4 is slightly different from the corresponding steps in \cite[Algorithm 3]{COV2014} where instead $\SFF_p^*(n_i)$ is computed for all $i=1,\ldots, q$ and then all these Hessian matrices are stacked, rather than computing $\SFF_p^*(\eta)$ which corresponds to randomly combining these matrices. The present approach is slightly more efficient.

In our implementation, we construct in S2 a matrix $T$ whose rows contain the tangent vectors. Then, in S3, we reduce the matrix $T$ to row Echelon form, so that the tangent space is spanned by the rows of
\( 
T \simeq \begin{bmatrix} I_{(3n+1)r} & X \end{bmatrix} P,
\)
where $P$ is a permutation matrix determined by the pivot selection during the reduction to Echelon form.
Therefore $\eta \in \langle N \rangle$ if and only if $\eta = P^{-1} (-X f_0, f_0)$ for any choice of $f_0 \in \CC^c$, where $c = \operatorname{codim} \widehat{\Var{C}}_{3,n} = \binom{n+3}{3} - (3n+1)r$. Consequently, the normal vectors can be parameterized via the free variables $f_0$.

In order to implement this algorithm reliably on a computer, we proceed as usual.
The points in S1 are chosen as $p_j = \prod_{i=1}^d (a_{0,i,j} x_0 + \cdots + a_{n,i,j} x_n)$ where the $a_{k,i,j}$ are sampled identically and independently distributed (i.i.d.) from the uniform distribution on $\ZZ_m = \{0, 1, \ldots, m-1\}$, where $m$ is some prime number. In step S3, we set $\eta = P^{-1}(-X f_0, f_0)$ by randomly sampling the elements of $f_0 \in \ZZ_m^{c}$ i.i.d.~from the uniform distribution.
The computations in foregoing algorithm are then all performed over the field $\ZZ/m\ZZ$. The validity of this standard approach, especially with respect to the genericity of the chosen points, was explained in \cite[Section 3]{COV2014}.

We implemented foregoing algorithm in C++ (only for $d=3$) by adapting earlier codes that were developed for \cite{COV2014,AV2018,TV2020}. The code was compiled with the GCC version 6.4.0 with \texttt{-O3 -DNDEBUG} flags active. All computations are performed in the prime field $\ZZ_{202001}$.\footnote{In our initial tests we had selected $\ZZ_{8191}$ as prime field as in \cite{AV2018,TV2020}, however for large $n$ this frequently led to failures ($T$ in step S2 failed to be of the correct dimension). For this reason, we increased the size of the prime field which quenched the problem. The specific prime chosen is the month in which the experiments were performed; January 2020.} We used \texttt{ReducedRowEchelonForm} from FFLAS--FFPACK \cite{FFLAS} for apply full Gaussian elimination to $T$. This package depends on Givaro and a BLAS implementation; we used OpenBLAS version 0.2.20.
The program code can be found in the ancillary files accompanying the arXiv preprint. 

Given the large sizes of the matrices involved, the computations for $n=2,\ldots,102$ were performed on KU Leuven/UHasselt's Tier-2 Genius cluster, which is part of the supercomputer of the Vlaams Supercomputer Centrum (VSC). 
Two different types of nodes were used. For $n=2,\ldots,90$, we used ``skylake'' nodes containing two Xeon\circledR{} Gold 6140 CPUs (18 physical cores, 2.3GHz clock speed, 24.75MB L3 cache) with 192GB of main memory, while for $n=91,\ldots,102$ we employed a ``bigmem'' node containing the same 2 Xeon\circledR{} Gold 6140 CPUs but with 768GB of main memory.\footnote{According to the web page \url{https://vlaams-supercomputing-centrum-vscdocumentation.readthedocs-hosted.com/en/latest/leuven/tier2_hardware/genius_hardware.html}} We requested all physical cores from the scheduling software so that OpenBLAS would have $36$ threads available for parallel processing.

% Specifically, the nodes displayed in \cref{tab_nodes} were used.
% \begin{table}[tb]
% \caption{List of Genius Tier-2 compute nodes that were used for proving not $r$-twd of $\Var{C}_n$.}
% \label{tab_nodes}
% \begin{tabular}{lcccccc}
% \toprule
% $n$ & $2, \ldots, 30$ & $31, \ldots, 60$ & $61,\ldots,70$ & $71,\ldots,80$ & $81,\ldots,90$ & $91,\ldots,102$ \\
% node & r22i27n04 & r22i13n12 & r22i27n13 & r22i27n18 & r22i27n24 & r23i27n15 \\
%  \bottomrule
% \end{tabular}
% \end{table}

The program generates certificates of not $r$-twd that record (1) the randomly chosen coordinates of $k_i, l_i, m_i$ (with respect to the standard basis $x_0, \ldots, x_n$) that define the points $p_i = k_i l_i m_i \in \widehat{\Var{C}}_{3,n}$ and (2) the vector $f_0 \in \ZZ_m^{c}$ of randomly chosen free variables. Applying the aforementioned algorithm to this configuration proves generic not $r$-twd. An example of such a certificate is shown below.\footnote{The full collection of certificates can be downloaded from the authors' web pages.}

{\small
\begin{verbatim}
Using random seed: 1591688259
k_0 = [17068  9508  8836  2681 14273  2196]
l_0 = [10549  3190 13747 17792 14579 19854]
m_0 = [ 3460  1587 17806  9155 16408 18933]
k_1 = [  328 11046  4677 16618 14053  1170]
l_1 = [ 2597  8062  6732   112 17180  6488]
m_1 = [13042   243 14543  8217  2423  5613]
k_2 = [ 2758   363 13376  9583  8315  5014]
l_2 = [19182  1662 19793  1788  5975 17021]
m_2 = [ 3018   609 15188 18700  1096 13016]
Constructed T in 0.001s.
Computed the rank of the 48 x 56 matrix T over F_20201 in 0s.
Found 48 vs. 48 expected.
f_0 = [ 5257  5355 19748  3457  1773 19861 15532 19684]
Constructed an element from the null space in 0.002s.
Constructed second fundamental form at k_0 l_0 m_0 in 0.001s.
Computed the rank of the 18 x 18 second fundamental form at 
    k_0 l_0 m_0 over F_20201 in 0s.
Found 15 vs. 15 expected.
5-nTWD is TRUE
Total computation took 0.005s.
\end{verbatim}}
 
The program verified for all $n = 2, \ldots, 102$ not $r$-twd holds at $r=\lceil \frac{1}{3n+1} \binom{n+3}{3} \rceil-1$.\footnote{For $n=1$, the generic rank is $\lceil \frac{4}{4} \rceil = 1$, and generic $1$-identifiability always holds.} Only three cases ($n = 71, 75, 86$) were not proved in the first round, due to unfortunate random choices (of points or element from the normal space). These were retested with a different random seed on skylake nodes, and then immediately found to be not $r$-twd, as expected. When compiling the final results, we noted that the certificate files for $n = 5, 6$ were corrupted, caused by an unknown problem. Therefore, the certificates were generated anew on a computer containing one Intel Core i7-4770K CPU (4 physical cores, 3.5GHz, 8MB L3 cache) and 32GB of memory. The cumulative time for establishing not $r$-twd of $\Var{C}_{3,n}$, counting only the successful runs, is displayed in \cref{fig_time}. The total time was about 9 days and 5 hours.

\begin{figure}[tb]
\caption{The cumulative time to prove not $\left(\lceil \frac{1}{3n+1} \binom{n+3}{3} \rceil-1\right)$-twd of the degree-$3$ Chow variety of $0$-cycles $\Var{C}_{3,k}$ for all $2 \le k \le n$.} 
\label{fig_time}
\includegraphics[width=\textwidth]{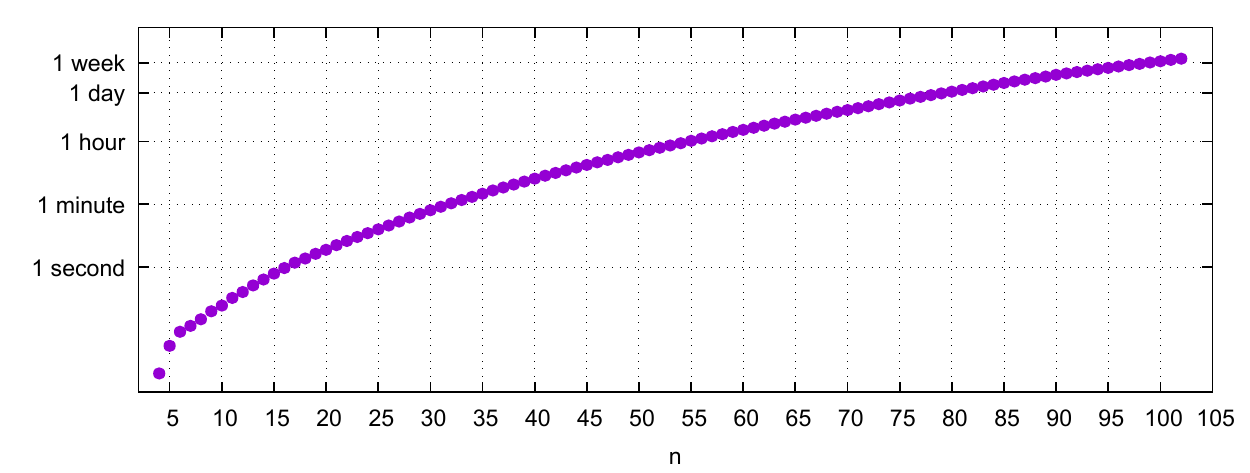} 
\end{figure}

Note that the largest case required to conclude the proof involved computations in $S^3 \CC^{103}$, whose dimension is $187460$. The largest computational burden occurs in S3, where a reduced echelon form of an approximately square matrix of size $\dim S^3 \CC^{n+1}$ should be computed. Recall that the cost of the naive algorithm is $\mathcal{O}(N^3)$ elementary operations over $\ZZ_m$. Fortunately, Strassen \cite{Strassen1969} first showed that faster algorithms exist, such as the slab recursive algorithm used by FFLAS--FFPACK \cite{JPS2013}. Indeed, the time increased only by about $7.204\%$ from 15h 34m 1s to 16h 41m 19s when going from $n=101$ to $102$, while the corresponding increase in memory consumption was $5.866\%$ from $265.28$GB to $280.84$GB. Then, $1.05866^{x/2} = 1.07204$ for $x \approx 2.44$, which is quite close to Coppersmith and Winograd's $\omega \approx 2.376$ exponent of matrix multiplication \cite{CW1990}.

\section{Some Riemannian geometry of the Chow variety} \label{sec_chow_not1wd}

Before we were aware of Oeding's result \cite{Oeding2012}, we had proved not $1$-weak defectivity of the Chow variety using a direct approach based on the Katz dimension formula \cite{GKZ1994}. Since this approach highlights some of the Riemannian and K\"ahlerian geometries of the smooth loci of the real and complex Chow varieties respectively, can clarify a few details about the tangent space required in \cref{sec_base_cases}, and furnishes an alternative proof of part of \cite[Theorem 1.3]{Oeding2012}, we decided to include it.

Recall that not $1$-twd of a projective variety $\Var{V}\subset\PP^N$ is implied by not $1$-weakly defectivity.
To show the latter, we can rely on the equivalence stated in Remark 3.1(ii) of \cite{CC2001} between not $1$-weak defectivity and the nondegeneracy of the Gauss map $\gamma : \Var{S} \to \mathbb{G}(d, N), \;x \mapsto \Tang{x}{\Var{V}}$, where $\Var{S}$ is the smooth locus of $\Var{V}$, $d=\dim\Var{V}$, and $\mathbb{G}$ is the projective Grassmannian. Note that $\gamma$ is nondegenerate, by semicontinuity of matrix rank, if there exists a point $p$ of $\Var{S}$ such that the derivative $\mathrm{d}_p \gamma$ is injective. Griffiths and Harris \cite{GH1979} explained how this linear map $\mathrm{d}_p \gamma : \Tang{p}{\Var{S}} \to \Tang{\Tang{p}{\Var{S}}}{\mathbb{G}(d, N)}$ can be interpreted as the \textit{projective second fundamental form} $|\SFF_p|$ at $p\in \Var{S}$.

We briefly recall the definition of the second fundamental form from K\"ahler geometry; for more details see \cite{GH1978,Huybrechts2005,KN1969,Kodaira1986}.  
A complex submanifold $\Var{M}\subset\CC^{N+1}$ of dimension $d$ is a $2d$-dimensional real-differentiable manifold along with a holomorphic atlas. The holomorphic tangent bundle $\Tang{}{\Var{M}}$ of $\Var{M}$ is a holomorphic vector bundle of rank $d$ with fiber at $p \in \Var{M}$ equal to the complex vector space
\[
 \Tang{p}{\Var{M}} := \left\{ \frac{\mathrm{d}}{\mathrm{d} z} \gamma_p(z) \mid \gamma_p(z) \subset \Var{M} \text{ is a holomorphic curve and } z \in \CC \right\} \subset \CC^{N+1},
\]
where $\mathrm{d}/\mathrm{d} z$ is the complex derivative.
By equipping the tangent space $\Tang{p}{\CC^{N+1}}$ with the Hermitian inner product $h_p( x, y ) := x^T \overline{y}$, $\CC^{N+1}$ becomes a \textit{K\"ahler manifold} \cite[Chapter 3]{Huybrechts2005}. 
A K\"ahler manifold admits a unique \textit{Chern connection}\footnote{This is a connection that is compatible with the Hermitian metric and the holomorphic structure of the tangent bundle.} $\nabla$ on the holomorphic tangent bundle \cite[Proposition 4.2.14]{Huybrechts2005}, and for $\CC^{N+1}$ it coincides with the usual exterior differential.
Equipping the submanifold $\Var{M} \subset \CC^{N+1}$ with the latter's K\"ahler metric $h$ turns it into a K\"ahler submanifold. The normal space $\Norm{p}{\Var{M}} \subset \CC^{N+1}$ of such a manifold at $p$ is defined as the orthogonal complement of the tangent space $\Tang{p}{\Var{M}}$ in $\CC^{N+1}$.

The second fundamental form $\SFF$ of $\Var{M} \subset \CC^{N+1}$ at $p\in\Var{M}$ is 
\[
 \SFF_p(X,Y) = \mathrm{P}_{\mathrm{N}_p \Var{M}} (\nabla_X Y)  \in \Norm{p}{\Var{M}}
\]
where $X,Y$ are sections of the holomorphic tangent bundle $\Tang{}{\Var{M}}$ of $\Var{M}$, and $\mathrm{P}_{\mathrm{N}_p \Var{M}}$ is an orthogonal projector onto the normal space $\Norm{p}{\Var{M}}$ of $\Var{M}$ at $p$. The second fundamental form is symmetric in that $\SFF_p(X,Y)=\SFF_p(Y,X)$ \cite{GH1978}, and by dualization it can be interpreted as the map $\SFF_p^* : \Norm{p}{\Var{M}} \to S^2(\Tang{p}{\Var{M}})$ \cite[Example 17.11]{Harris1992}.

The equivalent definition over the reals (for ordinary real submanifolds $\Var{M}\subset\RR^{N+1}$) is obtained by replacing $\CC$ by $\RR$, ``complex'' by ``real,'' ``K\"ahler'' by ``Riemannian,'' the Hermitian inner product by the Euclidean inner product $h_p(x,y)=x^T y$, ``holomorphic'' by ``analytic,'' and ``Chern'' by ``Levi--Civita'' in the above discussion; see \cite{Lee2013,Lee1997,doCarmo1993}. 

In the remainder of this section, we drop the subscript of $\Var{C}_{d,n}$, i.e., $\Var{C} = \Var{C}_{d,n}$.

\subsection{Minimality of the smooth locus of the Chow variety}
Let $\bbk = \RR$ or $\CC$.  
We establish minimality of the smooth locus of the Chow variety $\Var{C}$ by partially computing the second fundamental form.

Let $L_0, \ldots, L_{d-1} \in \bbk^{n+1}$ and take the standard coordinates $x_0, \ldots, x_n$ on $\bbk^{n+1}$. Then, the tangent space to $\widehat{\Var{C}} \subset S^d \bbk^{n+1}$ at the smooth point $p = L_0 \cdots L_{d-1}$ is given by 
\[
\Tang{p}{\widehat{\Var{C}}} = \langle x_0 L_1 \cdots L_{d-1}, \ldots, x_n L_1 \cdots L_{d-1}, \ldots, x_0 L_0 \cdots L_{d-2}, \ldots, x_n L_0 \cdots  L_{d-2}\rangle.
\]
Let $0 \le k < d$ and $0 \le i \le n$. The curve 
\[
 \gamma_{k,i}(t) = (L_k + t x_i) \prod_{0 \le \alpha \ne k < d} L_\alpha
\]
with $t\in \bbk$ is verified to pass through $p$ and is tangent along a vector field extension $E_{k,i}$ of the tangent bundle $\Tang{}{\gamma_{k,i}} = \langle x_i \prod_{\alpha\ne k} L_\alpha \rangle \subset \Tang{}{\widehat{\Var{C}}}$; hence $\gamma_{k,i}$ is the integral curve associated with this vector field $E_{k,i}$ through $p$.
It follows that
\[
E_{k,i}|_{\gamma_{l,j}(t)}
 = \begin{cases}
 x_i (L_l + t x_j) \prod_{0 \le \alpha \ne k, l < d} L_\alpha & \text{if } k \ne l,  \\
 x_i \prod_{0 \le \alpha \ne k < d} L_\alpha & \text{if } k = l.
 \end{cases}
\]

The second fundamental form of the smooth locus of $\widehat{\Var{C}}$ as a Riemannian ($\bbk=\RR$) and K\"ahlerian ($\bbk=\CC$) submanifold of $S^d \bbk^{n+1} \simeq \bbk^{\binom{n+d}{d}}$ with the standard inner product inherited from the latter space\footnote{Note that this is not the ``usual'' inner product weighted with multinomials that is typically used for $S^d \bbk^{n+1}$. The result is nevertheless the same.} is the projection of the directional derivative:
\begin{align}\label{eqn_sff}
 \SFF_p(E_{k,i}, E_{l,j}) 
 &= \mathrm{P}_{\mathrm{N}_p \widehat{\Var{C}}}\left( \frac{\mathrm{d}}{\mathrm{d}t}\Big|_{t=0} E_{k,i}|_{\gamma_{l,j}(t)} \right)
 = 
 \begin{cases}
  \mathrm{P}_{\mathrm{N}_p \widehat{\Var{C}}}\left( q_{i,j,k,l} \right) & \text{if } k \ne l, \\
  0 & \text{if } k = l,
 \end{cases}
\end{align}
where 
\(
q_{i,j,k,l} = x_i x_j \prod_{0\le\alpha \ne k,l < d} L_\alpha.
\)

Since $\SFF_p(E_{k,i}, E_{k,i}) = 0$ for all $0\le k< d$ and $0\le i\le n$, the trace of $\SFF_p$ is identically zero. This implies that the mean curvature is zero, and so we have proved the following result.

\begin{proposition}
The smooth locus of the real Chow variety is a minimal immersed Riemannian submanifold of $S^d(\RR^{n+1})$ for $d \ge 2$ and $n\ge0$.
\end{proposition}

\begin{remark}
 The equivalent statement for the smooth locus of the complex Chow variety holds as well by the foregoing argument. However, it is a corollary of the fact that every K\"ahlerian immersed submanifold is minimal by \cite[Theorem 3.1.2]{Simons1968}.
\end{remark}

\subsection{An alternative proof of Oeding's result}

The projective second fundamental form $|\SFF|$ of a projective submanifold $\Var{V}^n \subset\PP^N$ was defined in \cite{GH1979} as the linear system of quadrics formed by the restriction of the second fundamental form $\SFF$ of the cone $\widehat{\Var{V}}$ over $\Var{V}$ to the frame $(E_1, \ldots, E_{n})$, where $(E_0,\ldots,E_{n})$ is a frame for $\Tang{}{\widehat{\Var{V}}}$ such that $E_0|_p$ (the value of the vector field at $p$) lies over $[p]$ and $(E_0,\ldots,E_n)|_{p}$ spans $\Tang{p}{\widehat{\Var{V}}}$ at $p \in \widehat{\Var{V}}$. This leads to the following characterization. 

\begin{lemma}\label{lem_equiv_not1wd_sff}
Let $[p]$ be a smooth point of a reduced, irreducible, nondegenerate projective variety $\Var{V} \subset \PP^N$. If there exists a normal direction $x \in \Norm{p}{\widehat{\Var{V}}}$ such that $\SFF^*_p(x) \in S^2(\Tang{p}{\widehat{\Var{V}}})$ is invertible on $S^2(\Tang{[p]}{\Var{V}})$, then $\Var{V}$ is not $1$-weakly defective.
\end{lemma}
\begin{proof}
The conditions are equivalent to the existence of a nonsingular quadric in the linear system of quadrics formed by $|\SFF|$, so by (2.6) of \cite{GH1979} the Gauss map $\gamma$ is nondegenerate. The result follows from Remark 3.1(ii) of \cite{CC2001}.
\end{proof}

Based on this reformulation, which is essentially the Katz dimension formula  \cite{GKZ1994} for computing dimensions of the dual variety in the case the dimension is maximal, we can provide an alternative proof of part of Oeding's result.

\begin{proposition}[Part of Theorem 1.3 in \cite{Oeding2012}] \label{prop_n1wd_chow}
  The degree-$d$ Chow variety of zero-cycles $\Var{C}_{d,n}$ in $n+1$ variables is not $1$-weakly defective if $3 \le d \le n+1$.
\end{proposition}
\begin{proof}
Since $3 \le d \le n+1$, we can consider the smooth point $[p] = [x_0 \cdots x_{d-1}] \in \Var{C}$. In this case, the tangent space $\Tang{p}{\widehat{\Var{C}}}$ is spanned by a basis of monomials, namely by 
\begin{align*}
[p] &= \langle x_0 \cdots x_{d-1} \rangle, \\
A &= \langle x_2\cdots x_{d-1}, \ldots, x_1 \cdots x_{d-2} \rangle x_0^2 \oplus \cdots \oplus \langle x_1 \cdots x_{d-2}, \ldots, x_0 \cdots x_{d-3} \rangle x_{d-1}^2, \\
A' &= \langle x_{d}, \ldots, x_n \rangle x_1 \cdots x_{d-1} \oplus \cdots \oplus 
\langle x_{d}, \ldots, x_n \rangle x_0 \cdots x_{d-2}.
\end{align*}
These basis vectors are orthogonal with respect to the Euclidean and Hermitian inner products for $\RR$ and $\CC$ respectively. Consequently, 
\(
 \dim_\bbk [p] = 1,\, \dim_\bbk A = d(d-1), \text{ and } \dim_\bbk A' = d(n + 1 -d).
\)
Observe that \(A = \mathrm{span}\bigl( ( E_{k,i}|_p )_{0 \le k \ne i < d} \bigr)\), \(A' = \mathrm{span}\bigl( ( E_{k,i}|_p )_{0 \le k < d \le i \le n} \bigr)
\), and
\[
 E_{0,0}|_p = E_{1,1}|_p = \cdots = E_{d-1,d-1}|_p = p.
\]
The integral manifold associated with these last vector fields is exactly the fiber of the projection $\bbk^{N+1} \to \PP_\bbk^N, x \mapsto [x]$ at $p$. It follows that $\Tang{[p]}{\Var{C}} \simeq A \oplus A'$ is spanned by the orthogonal frame $\Var{E}$ formed by all $E_{k,i}$, $k=0,\dots,d-1$ and $i=0,\dots,n$, except for $E_{i,i}$:
\(
 \Var{E} = \Bigl( E_{k,i} \Bigr)_{\substack{0 \le k < d, 0 \le i \ne k \le n}}.
\)
For determining the \textit{projective} second fundamental form it suffices to compute the usual second fundamental form of (an open neighborhood of) the cone $\widehat{\Var{C}}$ at $\SFF_p(E,F)$ for all vector fields $E, F \in \Var{E}$.

Inspecting \cref{eqn_sff}, we have
\(
q_{i,j,k,l} = x_i x_j \prod_{0\le\alpha \ne k,l < d} x_\alpha.
\)
We project the monomials $q_{i,j,k,l}$ onto $\Norm{p}{\widehat{\Var{C}}}$.
Because $q_{j,i,k,l}=q_{i,j,k,l}$ and $q_{i,j,l,k}=q_{i,j,k,l}$ for all valid $0 \le i, j \le n$ and $0 \le k \ne l < d$, it suffices to determine what happens to the projection if 
$0 \le i \le j \le n$ and $0 \le k < l < d$. Note that this is easy with the given monomial bases, as it just consists of verifying whether or not $q_{i,j,k,l}$ is one of these basis vectors. After some computations, we obtain
\begin{align}\label{eqn_sff_components}
\SFF_p(E_{k,i}, E_{l,j}) = \begin{cases} 
 x_i^2 x_j^2 \prod_{0 \le \alpha\ne k,l,i,j < d} x_\alpha &\text{if } 0 \le i < j < d \text{ and } \sharp\{k,l,i,j\}=4 \\ %k\ne l \ne i \ne j, \\
 x_i^3 \prod_{0\le \alpha\ne k,l,i<d} x_\alpha & \text{if } 0 \le i = j < d \text{ and } \sharp\{k,l,i\}=3 \\ % k \ne l \ne i, \\
 x_i^2 x_j \prod_{0\le\alpha\ne k,l,i<d} x_\alpha & \text{if } 0 \le i < d \le j \text{ and } \sharp\{k,l,i\}=3 \\ %k \ne l \ne i, \\
 x_i x_j \prod_{0\le\alpha\ne k,l<d} x_\alpha & \text{if } d \le i \le j \text{ and } k \ne l, \\
 0 & \text{ otherwise},
 \end{cases}
\end{align}
where $\sharp S$ denotes the cardinality of the set $S$.

Consider the polynomial
\[
 \eta = \sum_{0 \le i\ne k\ne l < d} x_i^3 \prod_{0\le\alpha\ne k,l,i<d} x_\alpha + \sum_{j=d}^n x_j^2 \sum_{0 \le k \ne l < d} \,\prod_{0\le\alpha\ne k,l<d} x_\alpha \in S^{d} \CC^{n+1}.
\]
Note that it is a linear combination of nonzero monomials appearing in the right hand side of \cref{eqn_sff_components}, so that $\eta \in \Norm{p}{\widehat{\Var{C}}}$. 
Contracting $\eta$ with $\SFF_p(E_{k,i}, E_{l,j})$ yields
\begin{align*}%\label{eqn_matrix_entries}
 \langle \SFF_p(E_{k,i}, E_{l,j}), \eta\rangle = \begin{cases}
 1 & \text{if } 0 \le i = j < d \text{ and } \sharp\{k, l, i\} =3, \\
 1 & \text{if } d \le i = j \text{ and } k \ne l, \\
 0 & \text{otherwise}.
 \end{cases}
\end{align*}
Putting all of these together in a succinct matrix, with rows indexed by $(k,i)$ and columns by $(l,j)$, and the vector fields from the frame $\Var{E}$ in the order determined by the labeling of the rows and columns in the matrices below, we obtain
\[
|\SFF_p^*|(\eta) = 
 \begin{blockarray}{ccc}
 & 0 \le j < d & d \le j \le n \\
 \begin{block}{c[cc]}
 0 \le i < d & G & 0 \\
 d \le i \le n & 0 & H \\
 \end{block}
 \end{blockarray},
\]
where 
\[
 H = 
 \begin{blockarray}{ccccc}
 & l=0 & l=1 & \cdots & l=d-1 \\
 \begin{block}{c[cccc]}
 k=0 & 0 & I_{n+1-d} & \cdots & I_{n+1-d} \\
 k=1 & I_{n+1-d} & 0 & \ddots & \vdots \\
 \vdots & \vdots & \ddots & \ddots & I_{n+1-d} \\
 k=d-1 & I_{n+1-d} & \cdots & I_{n+1-d} & 0 \\
  \end{block}
 \end{blockarray} 
 = I_{n+1-d} \otimes \mathbf{1}_d - I_{d(n+1-d)}
\]
with $\mathbf{1}_d$ the $d \times d$ matrix filled with ones and $I_d$ the $d \times d$ identity matrix, and 
\begin{align*}
 G &= 
 \begin{blockarray}{ccccc}
 & j = 0 & j=1 & \cdots & j=d-1 \\
 \begin{block}{c[cccc]}
  i=0 & \mathbf{1}_{d-1} - I_{d-1} & 0 & \cdots & 0 \\
  i=1 & 0 & \mathbf{1}_{d-1} - I_{d-1} & \ddots & \vdots\\
  \vdots & \vdots & \ddots & \ddots & 0\\
  i=d-1& 0 & \cdots & 0 & \mathbf{1}_{d-1} - I_{d-1}\\
  \end{block}
 \end{blockarray} \\
 &= \mathbf{1}_{d-1} \otimes I_{d} - I_{d(d-1)}.
\end{align*}

Note that the submatrix of $|\SFF_p^*|(\eta)$ corresponding to the vector fields generating $A$ is $G$, and likewise for $A'$ and the submatrix $H$.
Hence, the diagonal blocks of $G$ are square matrices of order $d-1$. The reason is that for $0 \le i < d$ the vector field $E_{i,i} \not\in \Var{E}$ because it is in the fiber of the projection map; for the $i$th matrix on the block diagonal, the $d-1$ valid indices are $0 \le k \ne i < d$ and $0 \le l \ne j < d$.

Recall that the eigenvalues of the tensor product $A \otimes B$ of symmetric matrices $A$ and $B$ are equal to all products of the eigenvalues of $A$ and $B$; see, e.g., \cite{dSM2014}. Since $\mathbf{1}_d$ has eigenvalues $d$ and $0$ (with multiplicity $d-1$), it follows that $\lambda(H) = \{ d-1, -1 \}$, so $H$ is an $d(n+1-d) \times d(n+1-d)$ invertible matrix. Similarly, $\lambda(G) = \{ d-2, -1 \}$, so that $G$ is an $d(d-1)\times d(d-1)$ invertible matrix.
As a result, we conclude from \cref{lem_equiv_not1wd_sff} that the Chow variety $\Var{C}_{d,n}$ is not $1$-weakly defective when $3 \le d \le n+1$. In particular, it not $1$-twd. This concludes the proof.
\end{proof}

\end{document}